\documentclass[preprint,11pt,printbiddoi]{imsart}

\RequirePackage{amsthm,amsmath,amsfonts,amssymb}
\RequirePackage[numbers,sort&compress]{natbib}

\startlocaldefs
\usepackage[english]{babel}
\usepackage[utf8]{inputenc}

\usepackage{mathtools, soul}
\usepackage{graphicx,subfigure}
\usepackage{dsfont}
\usepackage{bigints}
\usepackage{ulem}
\usepackage{cancel}
\usepackage{enumitem}
\usepackage{hyperref}
\usepackage{fix-cm}
\usepackage{type1cm}

\usepackage{xcolor}

\theoremstyle{plain}
\newtheorem{thm}{Theorem}[section]
\newtheorem*{thm*}{Theorem}
\newtheorem{lem}[thm]{Lemma}

\newtheorem{cor}[thm]{Corollary}

\theoremstyle{definition}
\newtheorem{defn}[thm]{Definition}
\newtheorem{rem}[thm]{Remark}

\usepackage{ifthen}
\usepackage{stmaryrd,bbm}

\newcommand\independent{\protect\mathpalette{\protect\independent}{\perp}} 
\def\independent#1#2{\mathrel{\rlap{$#1#2$}\mkern2mu{#1#2}}}

\def\phi{\varphi}

\def\bee{\begin{eqnarray*}}
\def\ene{\end{eqnarray*}}

\endlocaldefs

\begin{document}

\begin{frontmatter}

\title{Four Function Theorems Through Optimal Transport and Entropic Analogs}
\runtitle{Four Function Theorems}
\runauthor{M. Madiman, J. Melbourne, and C. Roberto}

\begin{aug}

\author[A]{\fnms{Mokshay}~\snm{Madiman}\ead[label=e1]{madiman@udel.edu}}, 
\author[B]{\fnms{James}~\snm{Melbourne}\ead[label=e2]{james.melbourne@cimat.mx}}
\and
\author[C]{\fnms{Cyril}~\snm{Roberto}\ead[label=e3]{croberto@math.cnrs.fr}}
\address[A]{University of Delaware \printead[presep={,\ }]{e1}}

\address[B]{Centro de Investigación en Matemáticas \printead[presep={,\ }]{e2}}

\address[C]{ Université Paris Nanterre \printead[presep={,\ }]{e3}}
\end{aug}

\begin{abstract}
We study entropic versions of four-function inequalities and give transport proofs of the Euclidean Ahlswede–Daykin theorem and 
a four-function version of the Prékopa–Leindler inequality due to Cordero-Erausquin and Maurey. 
One of our main results shows that the sharp range for the four-function inequality for equally weighted power means of finite orders \(\alpha,\beta\)
is \(0\leq\alpha+\beta\leq2\), with the Ahlswede–Daykin min–max inequality obtained as a limit.
En route, we prove corresponding entropy inequalities under the quadratic optimal transport coupling, 
including an entropy power inequality for the minimum and maximum in Euclidean space.
We also study the stability of four-function inequalities under convolution with log-concave product densities, 
recovering the preservation of log-supermodularity established by Karlin and Rinott, and 
explain its application to a conditional entropy power inequality for log-supermodular random variables.
\end{abstract}



\end{frontmatter}

\section{Introduction}

We compare the Ahlswede--Daykin four-function theorem with a more recent four-function theorem due to Cordero-Erausquin and Maurey \cite{CM17:1} that resembles the classical Pr\'ekopa--Leindler inequality. The two theorems have a common form, with the coordinatewise minimum and maximum in the former replaced by complementary arithmetic means in the latter. In exploring this analogy, we give transport proofs of both theorems and obtain extensions involving nonlinear means.

More precisely, define the generalized mean, for $x,y>0$, $\lambda\in(0,1)$ and $\alpha\in\mathbb R\cup\{\pm\infty\}$, by
\[
M_\alpha^\lambda(x,y):=\left(\lambda x^\alpha+(1-\lambda)y^\alpha\right)^{1/\alpha},
\]
where the special cases $\alpha=0,\pm\infty$ are understood in the limit:
\[
M_0^\lambda(x,y)=x^\lambda y^{1-\lambda},\qquad
M_{-\infty}^\lambda(x,y)=\min(x,y),\qquad
M_\infty^\lambda(x,y)=\max(x,y).
\]
For vectors, the means act coordinatewise. On the positive orthant, the continuous  version of the Ahlswede--Daykin theorem  \cite{AD78}  due to Batty and Bollmann \cite{BB80} gives
$\int f_1\int f_2\leq\int f_3\int f_4$ under the assumption
\[
f_1(x)f_2(y)\leq
f_3(M_{-\infty}^\lambda(x,y))f_4(M_\infty^{1-\lambda}(x,y)),
\]
while Cordero-Erausquin and Maurey proved, among other things, the same conclusion under the assumption
\[
f_1(x)f_2(y)\leq
f_3(M_1^\lambda(x,y))f_4(M_1^{1-\lambda}(x,y)).
\]
Here the functions are nonnegative, measurable and integrable; see Theorem~\ref{thm: ADCEM} for a precise statement and the relation to the corresponding results on $\mathbb R^d$. The assumptions are therefore both of the form
\[
f_1(x)f_2(y)\leq
f_3(M_\alpha^\lambda(x,y))f_4(M_\beta^{1-\lambda}(x,y)).
\]
It is natural to ask which other pairs of means give the same integral inequality. For equal weights, we obtain the following sharp range.

\begin{thm}\label{thm:intro-power-means}
Let $d\geq1$ and $\alpha,\beta\in\mathbb R$ satisfy $0\leq\alpha+\beta\leq2$. If nonnegative measurable integrable functions $f_1,f_2,f_3,f_4$ on $(0,\infty)^d$ satisfy
\[
f_1(x)f_2(y)\leq
f_3(M_\alpha^{1/2}(x,y))f_4(M_\beta^{1/2}(x,y))
\qquad\text{for all }x,y\in(0,\infty)^d,
\]
then
\[
\left(\int f_1\right)\left(\int f_2\right)
\leq
\left(\int f_3\right)\left(\int f_4\right),
\]
where all integrals are over $(0,\infty)^d$. The range $0\leq\alpha+\beta\leq2$ is necessary for this implication to hold for all such functions, already in dimension one.
\end{thm}

This is the functional assertion of Corollary~\ref{cor:power-mean-strip}. The range includes all complementary pairs $(q,1-q)$ and reaches the Ahlswede--Daykin inequality in the min--max limit. For unequal weights $\lambda\in(0,1)$, our transport approach also gives the integral inequality for $\alpha,\beta\in[0,1]$ (Corollary~\ref{thm:cordero-maurey-general}), extending \cite{CM17:1} to nonlinear conditions and answering the authors' question about a transport proof of their result.

The proofs proceed through entropy inequalities under the quadratic optimal transport coupling, following the approach of \cite{GRST21}. A particularly simple case is the min--max inequality
\[
H(X\wedge Y)+H(X\vee Y)\geq H(X)+H(Y),
\]
where $X,Y$ have absolutely continuous laws with finite second moments and finite Shannon entropies, and are coupled by quadratic optimal transport (Theorem~\ref{thm:entropy-min-max-general}). Thus
\[
N(X\wedge Y)N(X\vee Y)\geq N(X)N(Y),
\]
where $N(X)=e^{2H(X)/d}$. This is an entropy power inequality for the join and meet in Euclidean space, but under the optimal transport coupling rather than the independent coupling. More generally, a scalar Jacobian criterion gives entropy inequalities for coordinatewise power means in arbitrary dimension (Theorem~\ref{thm:scalar-jacobian-entropy}). The functional inequalities follow by the variational formula for entropy.

We also study the stability of four-function inequalities under convolution, motivated by a recent investigation of super-additivity properties of entropy for dependent random variables \cite{MMR26}. It has long been known that the convolution of two log-supermodular densities need not be log-supermodular \cite[p.~486]{karlin1980classes}. We show that convolution with a log-concave product density preserves four-function inequalities. When the four functions coincide, this recovers the preservation of log-supermodularity established by Karlin and Rinott \cite[Proposition~3.7]{karlin1980classes}, which in particular answers the question about Gaussian convolution raised by Zartash and Robeva \cite{ZR22}.

Combining this preservation property with the results of \cite{MMR26} gives a conditional entropy power inequality. Namely, if $X,Y$ are real-valued random variables with log-supermodular joint density $p$, $\mathbb E|X|^2+\mathbb E|Y|^2<\infty$, and $p\log p\in L^1(\mathbb R^2)$, then
\[
e^{2H(X+Y)}\geq e^{2H(X\mid Y)}+e^{2H(Y\mid X)}.
\]
Here $H(X\mid Y)=H(X,Y)-H(Y)$. When $X$ and $Y$ are independent, this reduces to the classical entropy power inequality of Shannon.

Let us outline what remains. In Section~\ref{sec: definitions} we give definitions, notation, and a bit of background. Section~\ref{sec:Eucl-AD} studies the Euclidean Ahlswede--Daykin theorem and its entropic analog. We obtain the Riemann-integrable case from the finite-lattice theorem by truncation and Riemann sums, and give an entropy argument for general integrable functions. Section~\ref{sec: stability of conv} establishes preservation of four-function inequalities under convolution with log-concave product densities on $\mathbb R^d$ and $\mathbb Z^d$. Section~\ref{sec: LSM EPI} explains the application to the conditional entropy power inequality. In Section~\ref{sec: Transport} we prove the min--max entropy inequality by transport. We conclude in Section~\ref{sec: ADI PLI} with the scalar Jacobian criterion and its weighted and equal-weight power-mean consequences.

\section{Definitions} \label{sec: definitions}
When a non-negative function $f$ on $\mathbb{R}^d$ belongs to $L_1(\mathbb{R}^d)$, we will say it is a density.  If $\int f(z) dz = 1$, we will call $f$ a probability density.  We will abbreviate when convenient the Lebesgue integral of a function $f$ on $\mathbb{R}^d$, $\int_{\mathbb{R}^d} f(z) dz $ by just $\int f$.  A function $f$ is log-concave when $t \in (0,1)$ and $x,y \in \mathbb{R}^d$ imply that
\[
    f((1-t) x + ty) \geq f^{1-t}(x) f^t(y).
\]
We will call a density $f: \mathbb{R}^d \to [0,\infty)$ a product density if it admits a decomposition $f(z) = \prod_{i=1}^d f_i(z_i)$ for $f_i: \mathbb{R} \to [0,\infty)$.

\begin{defn}
A function $u \colon \mathbb{R}^d \to [0,\infty)$ is said to be log-supermodular if for all $x=(x_1,\dots,x_d), y=(y_1,\dots,y_d) \in \mathbb{R}^d$ it holds
$$
u(x) u(y) \leq u(x \wedge y) u(x \vee y) 
$$
where 
 $x \vee y \in \mathbb{R}^d$ denotes the componentwise maximum of $x$ and $y$ and $x \wedge y \in \mathbb{R}^d$
denotes the componentwise minimum of $x$ and $y$. Namely, if $x=(x_{1},\dots,x_{d})$ 
and $y=(y_{1},\dots,y_{d})$,
$$
(x \wedge y)_{i} = \min(x_{i},y_{i})
$$
and
$$
(x \vee y)_{i} = \max(x_{i},y_{i}).
$$
\end{defn}

Densities of this form have been  widely studied in 
various fields of mathematics-- they are sometimes referred to as MTP$_2$ (multivariate totally positive of order 2). The introduction of \cite{ZR22} provides a review of the literature. 

When $u$ is strictly positive, log-supermodularity of $u$ is equivalent to submodularity of the potential $V \coloneqq - \log u$, in the sense that
\[
    V(x) + V(y) \geq  V( x \wedge y) + V( x \vee y) .
\]
Submodular functions have been extensively studied in combinatorial optimization (see, e.g., \cite{Fuj05:book}) and are of importance to information theory (see, e.g., \cite{MT10, MK18}),
and more recently have found use in convex geometry \cite{FMMZ18, FMMZ22, FMZ22}.

For an Abelian group $G$ with Haar measure $dz$, and functions $f,g: G \to \mathbb{R}$ we write the usual convolution as $f*g(x) \coloneqq \int_G f(x-z) g(z) dz$ whenever this integral is well defined.  When the significance is clear from context, we will abbreviate integrals by $ \int f \coloneqq \int_G f(z) dz$.

We write $\mathcal P_2(\mathbb R^d)$ for the probability measures with finite second moment and $W_2$ for the quadratic Wasserstein distance:
\[
W_2(\mu,\nu)^2=\inf_{\pi\in\Pi(\mu,\nu)}
\int |x-y|^2\,d\pi(x,y),
\]
where $\Pi(\mu,\nu)$ denotes the set of couplings of $\mu$ and $\nu$. For a measurable map $T$, $T_\#\mu$ denotes its pushforward of $\mu$.

Let us also record the entropy conventions used throughout. In an
expression of the form $H(\mu\mid\lambda)$, $\lambda$ denotes Lebesgue
measure on the relevant Euclidean space.

\begin{defn} \label{def:shannon}
    For a random variable $X$ with probability density function $f$ on $\mathbb{R}^d$ we denote the Shannon entropy of $X$ by
    \[
        H(X) \coloneqq - \int f \log f 
    \]
    provided the integral is well defined.
\end{defn}

\begin{defn}\label{def:entropy-power}
For a random variable $X$ on $\mathbb{R}^d$ with well-defined Shannon entropy,
its entropy power is
\[
N(X)\coloneqq e^{\frac{2H(X)}{d}}.
\]
\end{defn}

In particular, if $X$ is a real-valued random variable, $N(X)$ is just $e^{2H(X)}$.

\begin{defn}\label{def:relative-entropy}
Let $\mu$ be a probability measure and let $\rho$ be a $\sigma$-finite
reference measure on the same measurable space. If $\mu\ll\rho$ and the
extended integral below is well defined, the relative entropy of $\mu$ with respect to
$\rho$ is
\[
H(\mu\mid\rho)
\coloneqq
\int \log\left(\frac{d\mu}{d\rho}\right)\,d\mu.
\]
We set $H(\mu\mid\rho)=+\infty$ when $\mu\not\ll\rho$. When $\rho$ is a
probability measure, this is the usual nonnegative relative entropy. When
$\rho=\lambda$ is Lebesgue measure, it need not be nonnegative. In
particular, if $\mu=f\,d\lambda$ is the law of $X$, then
\begin{equation}\label{eq:shannon-relative-sign}
H(\mu\mid\lambda)
=
\int f\log f\,d\lambda
=
-H(X),
\end{equation}
whenever these quantities are defined.
\end{defn}

\section{The Euclidean Ahlswede--Daykin 4-function theorem}
\label{sec:Eucl-AD}

The original four-function theorem of Ahlswede and Daykin \cite{AD78} is  classically stated on a finite distributive lattice.

\begin{thm}[Ahlswede--Daykin \cite{AD78}] \label{thm: discrete four function}
    If nonnegative functions $f_i$ on a finite distributive lattice satisfy
    \[
        f_1(x) f_2(y) \leq f_3(x \wedge y ) f_4(x \vee y)
    \]
    then
    \[
        \int f_1 \int f_2 \leq \int f_3 \int f_4,
    \]
    where $\int f_i$ denotes integration with respect to the counting measure.
\end{thm}

See \cite{AD79, Gra83, FS00, AS16:book} for more background on the inequality's significance and utility. There is also a $2m$-function generalization of the Ahlswede--Daykin theorem \cite{RS93} (see also \cite{RS92, AK96}); a different extension is considered in \cite{Pin19}.

Note that the celebrated FKG inequality \cite{FKG71} is a special case of the Ahlswede--Daykin theorem. Indeed, if $a,b$ are nonnegative and increasing and $\mu$ is log-supermodular, apply the theorem with
$f_1=a\mu$, $f_2=b\mu$, $f_3=\mu$, and $f_4=ab\mu$. On a finite lattice, general increasing functions can be shifted by constants to be nonnegative, without changing their covariance.
The continuous version of the FKG inequality was developed a long time ago \cite{Pre74} (see also \cite{Bat76, Edw78, BR80} where extensions to countable product spaces as well as various applications are detailed).

We will have use for the following Euclidean version, due to Batty and Bollmann \cite{BB80}, of the four-function theorem of Ahlswede and Daykin \cite{AD78}.  

\begin{thm}[Batty-Bollmann \cite{BB80}] \label{thm: four function continous}
For $L_1$ functions $f_i: \mathbb{R}^d \to [0,\infty)$ such that
\[
    f_1(x) f_2(y) \leq f_3(x \wedge y) f_4(x \vee y),
\]
it holds that
\[
       \int f_1 \int f_2 \leq \int f_3 \int f_4,
\]
where $\int f_i$ is integration with respect to the usual Lebesgue measure.
\end{thm}

A new and self-contained transport proof of Theorem \ref{thm: four function continous} is obtained below using the entropy inequality proved in Section \ref{sec: Transport}.
First we show, in the case that the $f_i$ are Riemann integrable, how Theorem \ref{thm: four function continous} follows from the finite version by truncation and convergence of Riemann sums.  

\begin{proof}[Sketch of proof of Theorem \ref{thm: four function continous}]
    Fix $n$ and, for $N\geq1$, set $\varepsilon_N=n/N$ and
    $L_N=\varepsilon_N\mathbb Z^d\cap[-n,n]^d$. The restrictions
    $\widetilde f_i=f_i|_{L_N}$ satisfy
    \[
        \widetilde f_1(x)\widetilde f_2(y)
        \leq \widetilde f_3(x\wedge y)\widetilde f_4(x\vee y)
    \]
    on the finite distributive lattice $L_N$. Hence Ahlswede--Daykin gives
    \[
        \left(\sum_{x\in L_N}\widetilde f_1(x)\right)
        \left(\sum_{x\in L_N}\widetilde f_2(x)\right)
        \leq
        \left(\sum_{x\in L_N}\widetilde f_3(x)\right)
        \left(\sum_{x\in L_N}\widetilde f_4(x)\right).
    \]
    Multiply by $\varepsilon_N^{2d}$ and let $N\to\infty$. Riemann
    integrability on $[-n,n]^d$ gives
    \[
        \int \mathbbm{1}_{[-n,n]^d} f_1 \int  \mathbbm{1}_{[-n,n]^d} f_2 \leq \int  \mathbbm{1}_{[-n,n]^d} f_3 \int  \mathbbm{1}_{[-n,n]^d} f_4.
    \]
    Since the truncation sets increase to $\mathbb R^d$, monotone
    convergence as $n\to\infty$ proves the theorem.
\end{proof}

Following \cite{GRST21}, who pioneered this approach in a discrete setting, we now give a proof of Theorem \ref{thm: four function continous} via
a new entropy inequality for coordinatewise minimum and maximum. First we state this entropy inequality, which we will prove in Section~\ref{sec: Transport}
via an optimal transport argument.

\begin{thm}
\label{thm:entropy-min-max-general}
Let $\nu_1,\nu_2\in\mathcal P_2(\mathbb R^d)$ be absolutely continuous probability measures. Let $\pi$ be their quadratic optimal transport coupling. Since $\nu_1$ is absolutely continuous, there exists a Brenier map
\[
T=\nabla\varphi
\]
such that
\[
\pi=(\operatorname{Id},T)_\#\nu_1.
\]
Define
\[
m_-(x,y)=x\wedge y,
\qquad
m_+(x,y)=x\vee y,
\]
where the minimum and maximum are taken coordinatewise, and set
\[
\nu_-=(m_-)_\#\pi,
\qquad
\nu_+=(m_+)_\#\pi.
\]
Then
\begin{equation}\label{eq:entropy-min-max-general}
H(\nu_-\mid\lambda)+H(\nu_+\mid\lambda)
\leq
H(\nu_1\mid\lambda)+H(\nu_2\mid\lambda),
\end{equation}
with values allowed in $(-\infty,+\infty]$.
\end{thm}

\begin{rem}\label{rem:min-max-entropy-power}
Let $(X,Y)$ have the quadratic optimal transport coupling $\pi$ from
Theorem~\ref{thm:entropy-min-max-general}. In terms of Shannon entropy,
\eqref{eq:entropy-min-max-general} is
\[
H(X\wedge Y)+H(X\vee Y)\geq H(X)+H(Y).
\]
Consequently, it is equivalent to the multiplicative
entropy-power inequality
\[
N(X\wedge Y)N(X\vee Y)\geq N(X)N(Y).
\]
It is in this sense that Theorem~\ref{thm:entropy-min-max-general} may be
viewed as an entropy power inequality for minimum and maximum. Unlike the
classical entropy power inequality, however, it concerns the optimal
transport coupling rather than the independent coupling. More broadly,
submodularity and supermodularity properties of information-theoretic
functionals under natural operations have proved useful elsewhere. For
example, the submodularity of joint entropy was essentially known to Shannon (see \cite{Fuj78} for an
explicit statement), and that of entropy under convolution was established in \cite{Mad08:itw} (see also \cite{MK18} for the locally compact abelian-group setting).
\end{rem}

\begin{proof}[Proof of Theorem \ref{thm: four function continous}]
   Let $f_1, f_2, f_3, f_4$ be functions satisfying the hypothesis of the theorem. If $\int f_1=0$ or $\int f_2=0$, the conclusion is immediate. If both integrals are positive, then so are $\int f_3$ and $\int f_4$: indeed, partition a positive-measure subset of $\{f_1>0\}\times\{f_2>0\}$ according to the coordinatewise order of $(x,y)$. On at least one of the resulting $2^d$ pieces, the map
   \[
   (x,y)\longmapsto(x\wedge y,x\vee y)
   \]
   is a coordinate permutation and the piece has positive measure. The pointwise hypothesis and Fubini's theorem then imply that both $\{f_3>0\}$ and $\{f_4>0\}$ have positive measure.

   Let $\nu_1,\nu_2$ be compactly supported absolutely continuous probability measures of finite entropy, concentrated respectively on $\{f_1>0\}$ and $\{f_2>0\}$, and such that $\log f_i\in L^1(\nu_i)$ for $i\in\{1,2\}$.  Let $\pi$ be their quadratic optimal coupling and define $\nu_-,\nu_+$ as in Theorem~\ref{thm:entropy-min-max-general}.
   The output measures have compact support and finite entropy by Theorem~\ref{thm:entropy-min-max-general}. The Gibbs bound on a compact set containing their supports, applied to the bounded potentials $\min\{\log(1+f_j),k\}$ and then letting $k\to\infty$, gives integrability of the positive parts of their logarithms. The pointwise hypothesis and the integrability of the input logarithms then give integrability of the negative parts. Thus all logarithmic integrals below are finite.
   Then, by definition of the coupling $\pi$, by hypothesis and by definition of $\nu_-,\nu_+$,  it holds
   \begin{align*}
   \int \log f_1 d\nu_1 + \int \log f_2 d\nu_2 
   & =
   \iint \log f_1 (x) + \log f_2 (y) \pi(dx,dy) \\
   & \leq 
      \iint \log f_3 (m_-(x,y)) + \log f_4 (m_+(x,y)) \pi(dx,dy) \\
      & =
      \int \log f_3 d\nu_- + \int \log f_4 d\nu_+ .
   \end{align*}
Subtracting the inequality of Theorem \ref{thm:entropy-min-max-general}, we end up with
\begin{align*}
  \left[\int \log f_1 d\nu_1 - H(\nu_1 |\lambda) \right] 
   + 
  \left[ \int \log f_2  d\nu_2 - H(\nu_2 |\lambda) \right] 
  & \leq 
   \left[\int \log f_3 d\nu_- - H(\nu_- |\lambda) \right] \\
  & \quad + 
  \left[ \int \log f_4  d\nu_+ - H(\nu_+ |\lambda) \right] .
\end{align*}
We use the following dual expression of the log-Laplace transform, for a measurable potential $\varphi$ with $0<\int e^\varphi\,d\lambda<\infty$:
\begin{equation}\label{eq:gibbs-lebesgue}
\log \int e^\varphi\,d\lambda
=\sup_\nu\left\{\int\varphi\,d\nu-H(\nu\mid\lambda)\right\},
\end{equation}
where the supremum is over absolutely continuous probability measures with compact support, finite entropy relative to $\lambda$, and $\varphi\in L^1(\nu)$.
Note that although the Donsker-Varadhan variational formula \cite[Lemma~2.1]{DV75}  was originally stated for probability measures,
it implies the corresponding statement for $\sigma$-finite reference measures as follows. Set
\[
Z=\int e^\varphi\,d\lambda,
\qquad
d\mu_\varphi=Z^{-1}e^\varphi\,d\lambda .
\]
For every admissible $\nu$, the identity
\[
H(\nu\mid\mu_\varphi)
=
H(\nu\mid\lambda)-\int\varphi\,d\nu+\log Z,
\]
and the nonnegativity of relative entropy give
\[
\int\varphi\,d\nu-H(\nu\mid\lambda)\leq\log Z.
\]
For the reverse bound, put $f=e^\varphi$ and set
\[
E_k=\{x\in[-k,k]^d:k^{-1}\leq f(x)\leq k\},
\qquad Z_k=\int_{E_k}f\,d\lambda.
\]
For every $k$ with $Z_k>0$, the probabilities $d\nu_k=Z_k^{-1}f\mathbbm{1}_{E_k}\,d\lambda$ are admissible and satisfy
\[
\int\log f\,d\nu_k-H(\nu_k\mid\lambda)=\log Z_k
\longrightarrow\log\int f\,d\lambda.
\]
The two terms on the right of the entropy inequality are bounded by
$\log\int f_3$ and $\log\int f_4$, respectively. Optimizing its left
side independently over $\nu_1$ and $\nu_2$ therefore gives
\[
\log\int f_1+\log\int f_2
\leq
\log\int f_3+\log\int f_4.
\]
Exponentiating proves the claimed four-function inequality.

\end{proof}

\section{Stability of Convolution} 
\label{sec: stability of conv}

The following lemma uses Theorem \ref{thm: four function continous} to provide sufficient conditions for a density function to preserve log-supermodularity.
\begin{lem} \label{lem: sufficient condition for log-supermodularity preservation}
    Let a density $g: \mathbb{R}^d \to [0,\infty)$ be such that
    \begin{equation} \label{eq: local condition for stable logsupermodularity}
        g(x-u) g(y - w) \leq g( x \wedge y - u \wedge w) g( x \vee y - u \vee w),
    \end{equation}
    for all $x,y,u,w \in \mathbb{R}^d$.
    If $f_i:\mathbb R^d\to[0,\infty)$, $i=1,2,3,4$, are measurable and integrable, each $f_i*g$ is finite everywhere, and
    \begin{align} \label{eq: f_i hypothesis}
        f_1(x)f_2(y) \leq f_3(x \wedge y ) f_4(x \vee y)
    \end{align} for all $x,y$ then
    \[
        f_1 * g(x) \  f_2*g(y) \leq f_3 * g(x \wedge y) \ f_4*g(x \vee y).
    \]
    In particular if $f$ is a log-supermodular density and $f*g$ is finite everywhere, then $f*g$ is log-supermodular.
\end{lem}

\begin{proof}
    Fix $x,y \in \mathbb{R}^d$ and define 
    \begin{align*}
        \tilde{f}_1(z) &\coloneqq f_1(z) g(x - z)
            \\
        \tilde{f}_2(z) &\coloneqq f_2(z) g(y-z) 
            \\
        \tilde{f}_3(z) &\coloneqq f_3(z) g(x \wedge y - z)
            \\
        \tilde{f}_4(z) &\coloneqq f_4(z) g(x \vee y - z).
    \end{align*}
    We wish to show that
    \[
        f_1*g(x) \  f_2*g(y) \leq f_3*g(x \wedge y ) \ f_4*g(x \vee y)
    \]
   which precisely means
    \[
        \int \tilde{f}_1 \int \tilde{f}_2 \leq \int \tilde{f}_3 \int \tilde{f}_4, 
    \]
    so that by Theorem \ref{thm: four function continous} it suffices to check that
    \[
        \tilde{f}_1(u) \tilde{f}_2(w) \leq \tilde{f}_3( u \wedge w) \tilde{f}_4(u \vee w).
    \]
    This is exactly 
    \[
    f_1(u) g(x - u) f_2(w) g(y - w) \leq f_3(u \wedge w) g(x \wedge y  - u \wedge w) f_4( u \vee w)  g(x \vee y - u \vee w),
    \]
    Thus, using \eqref{eq: f_i hypothesis} the conclusion follows from
    \[
    g(x - u) g(y - w) \leq g(x \wedge y  - u \wedge w)  g(x \vee y - u \vee w),
    \]
    which is exactly our hypothesis.  Taking $f_i = f$ for a log-supermodular function, shows the stability of log-supermodularity under log-concave product convolutions.
\end{proof}
Note that \eqref{eq: local condition for stable logsupermodularity} is equivalent to $u:(\mathbb{R}^d)^2 \to [0,\infty)$ defined by
\[
    u(x,y) = g(x-y)
\]
being log-supermodular. Write $V=\log g$ and suppose for this heuristic that $g$ is positive and smooth. Applying the differentiable characterization of supermodularity \cite{Topkis1978} to
$\log u(x,y)=V(x-y)$, and differentiating in the directions $e_i(x)$ and $e_j(x)$ for $i\neq j$, gives
\[
    \frac{\partial^2}{\partial x_i \partial{x_j}}  \log u(x,y) = \frac{\partial^2V}{\partial x_i \partial x_j} \geq 0
\]
while differentiating in the directions $e_i(x)$ and $e_j(y)$ gives
\[
   -\frac{\partial^2V}{\partial x_i \partial x_j} \geq 0.
\]
Thus the mixed partials of $V$ vanish. It follows that there exist $V_i: \mathbb{R} \to \mathbb{R}$ and a constant $c$ such that
\[
    V(x) = c+V_1(x_1) + \cdots + V_d(x_d).
\]
Moreover, differentiating in $e_i(x)$ and $e_i(y)$ gives $-V_i''\geq0$, so each $V_i$ is concave. This suggests that, at least as a consequence of Ahlswede--Daykin, $g$ being a log-concave product density is a necessary condition for $g$ to preserve inequality \eqref{eq: local condition for stable logsupermodularity} on convolution. The following result shows that this condition is sufficient.
\begin{thm} \label{thm: logconcave products preserve LSM}
    Let $g$ be a log-concave product density on $\mathbb{R}^d$. If $f_i:\mathbb R^d\to[0,\infty)$, $i=1,2,3,4$, are measurable and integrable and satisfy
    \begin{align*}
        f_1(x)f_2(y) \leq f_3(x \wedge y ) f_4(x \vee y)
    \end{align*} for all $x,y$ then
    \[
        f_1 * g(x) \  f_2*g(y) \leq f_3 * g(x \wedge y) \ f_4*g(x \vee y).
    \]

\end{thm}
Note that in the case all $f_i$ are equal to a single function $f$, this recovers the classical result of Karlin-Rinott \cite[Proposition~3.7]{karlin1980classes} that log-concave\footnote{The Karlin-Rinott result is stated equivalently in terms of P\'olya frequency functions, explicitly that log-concave densities coincide with the class PF$_2$.} product convolutions preserve log-supermodularity and answers a recent question of Zartash and Robeva \cite{ZR22}.
\begin{proof}
    Since an integrable log-concave function on $\mathbb{R}$ is bounded, the product density $g$ is bounded. Hence each $f_i*g$ is finite everywhere.
    By Lemma \ref{lem: sufficient condition for log-supermodularity preservation} it suffices to check \eqref{eq: local condition for stable logsupermodularity}.  Since $g$ is a product density, given by $g(z) = \prod_{i=1}^d g_i(z_i)$ for $g_i$ log-concave, \eqref{eq: local condition for stable logsupermodularity} is exactly
    \[
        \prod_{i=1}^d g_i(x_i - u_i)g_i (y_i - w_i) \leq  \prod_{i = 1}^d g_i(x_i \wedge y_i  - u_i \wedge w_i)  g_i(x_i \vee y_i - u_i \vee w_i).
    \]
    Thus it suffices to prove \eqref{eq: local condition for stable logsupermodularity} when $g$ is a log-concave density on $\mathbb{R}$. Write $g=e^{-V}$, where $V:\mathbb{R}\to\mathbb{R}\cup\{+\infty\}$ is convex; the value $+\infty$ encodes the complement of the interval $\{g>0\}$. It is equivalent to prove
    \begin{equation} \label{eq: convex function statement for LSM}
        V(x- u) + V(y -w) \geq V(x \wedge y - u \wedge w) + V( x \vee y - u \vee w). 
    \end{equation}
    To this end, without loss of generality, we may assume $x \leq y$ and note that if $u \leq w$ \eqref{eq: convex function statement for LSM} is equality.  It remains to assume $u > w$ and thus prove
    \[
        V(x - u) + V(y-w) \geq V(x - w) + V( y - u).
    \]
    Set $a=x-u$, $b=y-w$, $c=x-w$, and $d=y-u$. Then
    \[
        a+b=c+d,
        \qquad
        |b-a|=(y-x)+(u-w)\geq |(y-x)-(u-w)|=|d-c|.
    \]
    Heuristically, for convex functions, spreading out points with the same sum increases the sum of their values. If either $V(a)$ or $V(b)$ is infinite, the desired inequality is immediate. Otherwise, the displayed relations imply that there exists $\theta\in[0,1]$ such that
    \[
        c=\theta a+(1-\theta)b,\qquad d=(1-\theta)a+\theta b.
    \]
    Convexity therefore gives $V(c)+V(d)\leq V(a)+V(b)$, which is the desired result.
\end{proof}

Since the standard Gaussian is both log-concave and a product density, this yields the relevant stability under Gaussian convolution used in Corollary~\ref{cor:conditional-epi}.

\begin{cor} \label{cor: ZR}
    For $g$ the standard Gaussian density on $\mathbb{R}^d$, and $f$ a log-supermodular density,
    $f*g$ is also log-supermodular.
\end{cor}
A scaling argument gives the corresponding statement for every centered
isotropic Gaussian density. Indeed, if $g$ is standard Gaussian and
$g_t(x)=t^{-d}g(x/t)$, then
\[
(f*g_t)(x)=(\widetilde f*g)(x/t),
\qquad \widetilde f(z)=f(tz),
\]
and $\widetilde f$ is log-supermodular whenever $f$ is. Conversely, if
$f*g_t$ is log-supermodular for every $t>0$, then $f*g_t\to f$ at every
Lebesgue point of $f$. Passing to the limit in the defining inequality
gives
\[
f(x)f(y)\leq f(x\wedge y)f(x\vee y)
\]
whenever all four displayed arguments are Lebesgue points. This holds
for almost every $(x,y)$: after partitioning according to the
coordinatewise order of $x$ and $y$, each of the four maps is a
coordinate projection, and the inverse image of a null set under such
a projection is null. Thus $f$ is log-supermodular in the almost
everywhere sense.

    The same definitions apply on the $d$-dimensional integer lattice.  We consider $f: \mathbb{Z}^d \to [0,\infty)$ to be log-supermodular when $f(x)f(y) \leq f( x \wedge y ) f(x \vee y )$ holds for all $x,y \in \mathbb{Z}^d$, considered as a subset of $\mathbb{R}^d$.
\begin{defn}
    A function $g: \mathbb{Z} \to [0,\infty)$ is log-concave if $\{ g > 0 \} = I \cap \mathbb{Z}$ for some (possibly infinite) interval $I \subseteq \mathbb{R}$ and
    \[
        g^2(n) \geq g(n+1) g(n-1) \qquad \qquad \forall n \in I\cap\mathbb Z.
    \]
\end{defn}
This definition on $\mathbb{Z}$ is equivalent to the existence of a log-concave function on $\mathbb{R}$ interpolating $g$. Similarly, we will consider $V: \mathbb{Z} \to \mathbb{R} \cup \{+\infty\}$ to be convex if there exists an extended-valued convex function on $\mathbb{R}$ interpolating $V$.

A function $g : \mathbb{Z}^d \to [0,\infty)$ will be called a $\mathbb{Z}^d$ product density if it admits a decomposition into $d$ functions $g_i :\mathbb{Z} \to [0,\infty)$ such that 
\[
    g(z) = \prod_{i=1}^d g_i(z_i).
\]

\begin{lem} \label{lem: discrete sufficient condition for log-supermodularity preservation}
    Let $g: \mathbb{Z}^d \to [0,\infty)$ be summable and satisfy
    \begin{equation} \label{eq: discrete local condition for stable logsupermodularity}
        g(x-u) g(y - w) \leq g( x \wedge y - u \wedge w) g( x \vee y - u \vee w),
    \end{equation}
    for all $x,y,u,w \in \mathbb{Z}^d$.
    If $f_i: \mathbb{Z}^d\to[0,\infty)$ are summable and satisfy
    \begin{align*}
        f_1(x)f_2(y) \leq f_3(x \wedge y ) f_4(x \vee y)
    \end{align*} for all $x,y$ then
    \[
        f_1 * g(x) \  f_2*g(y) \leq f_3 * g(x \wedge y) \ f_4*g(x \vee y).
    \]

\end{lem}

\begin{proof}
    The proof follows identically to Lemma~\ref{lem: sufficient condition for log-supermodularity preservation} with sums replacing integrals.
\end{proof}

\begin{thm} \label{thm: ZR}
    Let $g(z) = \prod_{i=1}^d g_i(z_i)$ be a summable $\mathbb{Z}^d$ product density whose factors $g_i$ are log-concave. If $f_i:\mathbb{Z}^d\to[0,\infty)$, $i=1,2,3,4$, are summable and 
    \[
        f_1(x) f_2(y) \leq f_3( x \wedge y) f_4(x \vee y)
    \]
    for all $x,y\in\mathbb Z^d$, then
    \[
        f_1 *g (x) f_2* g(y) \leq f_3*g( x \wedge y)\ f_4*g(x \vee y)
    \]
    for all $x,y\in\mathbb Z^d$.
\end{thm}

\begin{proof}
    The proof follows identically to Theorem~\ref{thm: logconcave products preserve LSM}.
\end{proof}

\section{Log-supermodularity and entropy power inequality} \label{sec: LSM EPI}
In this section we sketch the connection between the stability of log-supermodularity under Gaussian convolution and a conditional entropy power inequality explored in \cite{MMR26}.
We direct the interested reader to \cite{MMR26}, where we consider 
a collection of $n$ random variables taking values in $\mathbb{R}^d$.  For simplicity of exposition, 
here we take random variables $(X,Y) \in \mathbb{R}^2$  with a joint density $p$ and define the conditional entropy of $X$ given $Y$ by
\[
    H(X|Y) \coloneqq H(X,Y) - H(Y).
\]
Typical regularity conditions are assumed in \cite{MMR26} and throughout this section to justify some exchange of limits and integration by parts (for the de Bruijn formula).  
Namely we assume $\mathbb{E}|X|^2 + \mathbb{E}|Y|^2 < \infty$ and finite entropy, i.e., the joint density $p$ of $(X,Y)$ satisfies $p \log p \in L_1(\mathbb{R}^2)$.

We recall the entropy power inequality. Shannon formulated the inequality
in his seminal work \cite{Sha48}; rigorous proofs were subsequently given
by Stam \cite{Stam59} and Blachman \cite{Bla65}.
\begin{thm}[Entropy power inequality] \label{thm: Shannon EPI}
For independent $\mathbb{R}$-valued random variables $X,Y$ with densities and well-defined entropies $H(X),H(Y),H(X+Y)$,
\[
    N(X+Y) \geq N(X) + N(Y).
\]
\end{thm}

With the intention of extending some form of the entropy power inequality to dependent random variables, by studying the behavior of the conditional entropy, it will be convenient to define in the presence of a joint density for $(X,Y)$, a conditional entropy power,
\[
    N(X|Y) \coloneqq e^{2 H(X|Y)} .
\]
In \cite{MMR26} a (classical) semigroup approach is used to derive an integral inequality based on Fisher information bounds that extends previous work of Johnson \cite{Joh04:1}. 
In more detail, dropping the nonpositive terms involving differences between joint and marginal Fisher information in \cite[Corollary~5.2]{MMR26} gives the 
following ``conditional entropy power inequality'', with a deficit term controlled by an integral of a functional of the Fisher information. Set
\[
    \lambda
    =
    \frac{N(X\mid Y)}{N(X\mid Y)+N(Y\mid X)}\in(0,1).
\]
Let $Z_1$ and $Z_2$ be independent (of each other and of $(X,Y)$) standard Gaussian random variables and, for $s>0$, define
\begin{align}\label{eq: 2 d OU flow}
    X_s &= \frac{e^{-s}}{\sqrt{\lambda}} X+ \sqrt{1 - e^{-2s}}Z_1
        \\
    Y_s &= \frac{e^{-s}}{\sqrt{1-\lambda}} Y+ \sqrt{1 - e^{-2s}}Z_2 \nonumber,
\end{align}
and let $p_s$ denote the joint density of $(X_s,Y_s)$.
\begin{thm}[Madiman--Melbourne--Roberto]
For $\mathbb{R}$-valued random variables $X$ and $Y$, let $\lambda$ be as above. Then
\[
    e^{S} \ N(X + Y) \geq N(X|Y) + N(Y|X),
\]
where in the notation of \eqref{eq: 2 d OU flow},
\begin{align} \label{eq: S def}
S \coloneqq 4 \sqrt{\lambda(1-\lambda)} \int_0^\infty  \mathbb{E} (\partial_x \log p_s) (\partial_y \log p_s)(X_s,Y_s)  ds.
\end{align}
\end{thm}

We obtain the conditional entropy power inequality stated in the introduction as a corollary.

\begin{cor}\label{cor:conditional-epi}
    For $(X,Y)$ random variables with log-supermodular joint density,
    \[
        N(X+Y) \geq N(X|Y) + N(Y|X).
    \]
\end{cor}
Note that when $(X,Y)$ are independent $N(X|Y) = N(X)$ and $N(Y|X) = N(Y)$, one recovers Theorem \ref{thm: Shannon EPI}.

\begin{proof}
    We consider the integrand in \eqref{eq: S def} and integrate by parts,
    \begin{align*}
        \mathbb{E} (\partial_x \log p_s) (\partial_y \log p_s)(X_s,Y_s) 
            &=
                \int_{\mathbb{R}^2} (\partial_x \log p_s)(\partial_y \log p_s)p_s
                    \\
            &=
                \int_{\mathbb{R}^2} \partial_x p_s (\partial_y \log p_s)
                    \\
            &=
                - \int_{\mathbb{R}^2} (\partial_{xy} \log p_s) p_s.
    \end{align*}
    Since $\log p_s$ is smooth, for any $s$ such that $p_s$ is log-supermodular, Topkis' differentiable characterization gives $(\partial_{xy} \log p_s) \geq 0$ and hence the integrand is nonpositive. By a coordinatewise scaling argument, if $p$, the density of $(X,Y)$, is log-supermodular, then for $s>0$ the density of
    \[
        \left(\frac{e^{-s}}{\sqrt{\lambda}}X,
        \frac{e^{-s}}{\sqrt{1-\lambda}}Y\right)
    \]
    is log-supermodular. Observing that $\sqrt{1 - e^{-2s}} (Z_1,Z_2)$ has a log-concave product density and applying Theorem \ref{thm: logconcave products preserve LSM}, we have that $p_s$, the density of
    \[
        (X_s,Y_s)
        =
        \left(\frac{e^{-s}}{\sqrt{\lambda}}X,
        \frac{e^{-s}}{\sqrt{1-\lambda}}Y\right)
        + \sqrt{1 - e^{-2s}} (Z_1,Z_2),
    \]
    is log-supermodular. Hence
    \[
        \mathbb{E} (\partial_x \log p_s) (\partial_y \log p_s)(X_s,Y_s)  \leq 0
    \]
    for every $s >0$.  Integrating we have $S \leq 0$, and $e^S \leq 1,$ completing the proof.
\end{proof}

\section{A transport proof of the min--max entropy inequality} 
\label{sec: Transport}

In this section we give a transport proof of Theorem \ref{thm: four function continous}, following \cite{GRST21}.

We first isolate some ingredients that will be used in the proof.
The first is the injectivity of min/max maps under strict monotonicity.

\begin{lem}
\label{lem:min-max-injective}
Let $T\colon\mathbb R^d\to\mathbb R^d$ be strictly monotone, in the sense that
\[
\langle T(x)-T(y),x-y\rangle>0
\qquad\text{whenever }x\neq y.
\]
Then the maps
\[
T_-(x)=x\wedge T(x),
\qquad
T_+(x)=x\vee T(x)
\]
are injective.
\end{lem}

\begin{proof}
Suppose first that
$T_-(x)=T_-(y)=z$.
Since
$z=x\wedge T(x)=y\wedge T(y)$,
we have, coordinatewise,
$x-z\geq0$, 
$T(x)-z\geq0$, and 
\[
(x_i-z_i)(T_i(x)-z_i)=0
\]
for every $i$. 
Therefore,
\[
\langle x-z,T(x)-z\rangle=0.
\]
The same identities hold with $y$ in place of $x$. It follows that
\begin{align*}
\langle x-y,T(x)-T(y)\rangle
&=
\big\langle
(x-z)-(y-z),(T(x)-z)-(T(y)-z)
\big\rangle\\
&=
-\langle x-z,T(y)-z\rangle
-\langle y-z,T(x)-z\rangle\\
&\leq0.
\end{align*}
Strict monotonicity implies $x=y$. Thus $T_-$ is injective.

The proof for $T_+$ is the same, with all coordinatewise inequalities
reversed.
\end{proof}

We need a lemma allowing approximation by regular densities.
Let $\gamma$ denote the standard Gaussian probability measure on
$\mathbb R^d$:
\[
d\gamma(x)
=
(2\pi)^{-d/2}e^{-|x|^2/2}\,dx.
\]

\begin{lem}
\label{lem:regular-entropy-approximation}
Let $\nu\in\mathcal P_2(\mathbb R^d)$ with $H(\nu\mid\gamma)$  finite.
There exist bounded, smooth, uniformly convex domains $K_j$ and
probability measures $\nu^j$ such that:

\begin{enumerate}
\item each $\nu^j $ has a smooth density on $K_j$, bounded above and
below by positive constants;

\item
$\nu^j\longrightarrow\nu$  in $W_2$;

\item
$H(\nu^j\mid\gamma)\longrightarrow H(\nu\mid\gamma)$.
\end{enumerate}
For any finite collection of measures, the approximations may be
chosen on the same domains $K_j$.
\end{lem}

\begin{proof}
Write $d\nu=h\,d\gamma$.
For $R>0$, let $a_R=\nu(B_R)$ and define the normalized restriction
\[
d\nu^R
=
\frac{\mathbbm{1}_{B_R}}{a_R}\,d\nu = \frac{\mathbbm{1}_{B_R}h}{a_R}
\,d\gamma .
\]
Since $\nu\in\mathcal P_2(\mathbb R^d)$,
$\nu^R\longrightarrow\nu$  in $W_2$.
Moreover,
\begin{align*}
H(\nu^R\mid\gamma)
= \int_{B_R} \frac{h}{a_R}
\log\left(\frac{h}{a_R}\right)\,d\gamma
= \frac{1}{a_R}\int_{B_R}h\log h\,d\gamma -\log a_R.
\end{align*}
Since $a_R\to1$ and $h\log h$ is integrable, it follows that
\[
H(\nu^R\mid\gamma)
\longrightarrow
H(\nu\mid\gamma).
\]

Fix $R$ and choose the larger ball
\[
K=B_{R+1}.
\]
Write $q=d\nu^R/d\lambda$, extended by zero outside $B_R$. For
$M\geq1$, set
\[
q_M=\frac{q\wedge M}{\int_K(q\wedge M)\,d\lambda}.
\]
Then $q_M\to q$ in $L^1(K)$ and
\[
\int_Kq_M\log q_M\,d\lambda
\longrightarrow
\int_Kq\log q\,d\lambda.
\]
Indeed, the positive part follows by monotone convergence before
normalization, while the negative part is dominated on the bounded
set $K$ by the bound $t\log t\geq-e^{-1}$.

Let $\rho_\delta$ be a smooth mollifier supported in $B_\delta$ and
choose $\delta<1$. For fixed $M$, the functions
$g_{M,\delta}=q_M*\rho_\delta$ are supported in $K$, converge to $q_M$
in $L^1$, and are uniformly bounded. Hence, as $\delta\downarrow0$,
\[
\int_K g_{M,\delta}\log g_{M,\delta}\,d\lambda
\longrightarrow
\int_Kq_M\log q_M\,d\lambda.
\]
For each $j$, first choose $M_j\geq j$ so that truncation changes both
the $L^1$ distance and the entropy by at most $j^{-1}$, and then choose
$0<\delta_j<j^{-1}$ so that mollification contributes at most another
$j^{-1}$ to each error. The densities
$g_j=g_{M_j,\delta_j}$ then converge to $q$ both in $L^1(K)$ and in
Lebesgue entropy.
Replace $g_j$ by
\[
\widetilde g_j
=
(1-\varepsilon_j)g_j
+
\frac{\varepsilon_j}{|K|},
\qquad
\varepsilon_j\downarrow0.
\]
For fixed $j$, the $L^1$ change is at most $2\varepsilon_j$, and the
entropy of $\widetilde g_j$ tends to that of $g_j$ as
$\varepsilon_j\downarrow0$ by dominated convergence, since $g_j$ is
bounded and $K$ has finite measure. We may therefore choose
$\varepsilon_j$ so that each change is at most $j^{-1}$. Then
$\widetilde g_j$ is smooth on
$\overline K$ and bounded above and below by positive constants. Since
the perturbations tend to zero, the $L^1$ and entropy convergence are
unchanged. The $L^1$ convergence also implies $W_2$ convergence here,
because all measures are supported in $K$ and
\[
W_2^2(p\,d\lambda,q\,d\lambda)
\leq \frac{\operatorname{diam}(K)^2}{2}\|p-q\|_{L^1(K)}.
\]
Finally, on $K$,
\[
H(\mu\mid\gamma)=H(\mu\mid\lambda)
+\frac12\int_K|x|^2\,d\mu(x)+\frac d2\log(2\pi),
\]
so $L^1$ convergence on the bounded set $K$ also gives convergence of
the remaining moment term. Thus the relative entropies converge as
claimed.

Choose $R_k\uparrow\infty$. On the common ball
\[
K_k=B_{R_k+1}
\]
choose the $k$-th smoothing $\nu^k$ so that
\[
W_2(\nu^k,\nu^{R_k})\leq k^{-1},
\qquad
|H(\nu^k\mid\gamma)-H(\nu^{R_k}\mid\gamma)|\leq k^{-1};
\]
the fixed-$R$ convergence proved above makes this possible. Together
with the convergence of $\nu^{R_k}$, these inequalities give assertions
(2) and (3). For finitely many input measures, use the same $R_k$ and
choose each of the finitely many smoothings to satisfy these two bounds
on the common domain $K_k$.
\end{proof}

We observe that the  min--max marginals are stable under such approximation.

\begin{lem}
\label{lem:min-max-transport-stability}
Suppose that
$\nu_1^j\longrightarrow\nu_1$ and $\nu_2^j\longrightarrow\nu_2$
in $W_2$, and let $\pi_j$ be an optimal coupling between $\nu_1^j$ and $\nu_2^j$. 
If the optimal coupling $\pi$ between $\nu_1$ and $\nu_2$ is unique, then
$\pi_j\longrightarrow\pi$ in $W_2(\mathbb R^d\times\mathbb R^d)$.
Consequently,
\[
(m_-)_\#\pi_j\longrightarrow(m_-)_\#\pi,
\qquad
(m_+)_\#\pi_j\longrightarrow(m_+)_\#\pi
\]
in $W_2$.
\end{lem}

\begin{proof}

By the stability of optimal transport plans
\cite[Theorem~5.20]{Vil09:book}, every weak limit point of
$(\pi_j)$ is an optimal coupling of $\nu_1$ and $\nu_2$. By the
assumed uniqueness of that coupling,
\[
\pi_j\rightharpoonup\pi.
\]
Moreover,
\[
\int_{\mathbb R^{2d}}\bigl(|x|^2+|y|^2\bigr)\,d\pi_j(x,y)
=
\int_{\mathbb R^d}|x|^2\,d\nu_1^j(x)
+
\int_{\mathbb R^d}|y|^2\,d\nu_2^j(y),
\]
and the right-hand side converges to the corresponding expression for
$\pi$. The characterization of Wasserstein convergence
\cite[Theorem~6.9]{Vil09:book} therefore yields
\[
\pi_j\longrightarrow\pi
\qquad\text{in }W_2(\mathbb R^{2d}).
\]

The coordinatewise maps
\[
m_-(x,y)=x\wedge y,
\qquad
m_+(x,y)=x\vee y
\]
are $1$-Lipschitz from $\mathbb R^{2d}$ to $\mathbb R^d$. Since
pushforward by an $L$-Lipschitz map is $L$-Lipschitz in Wasserstein
distance
\cite[Lemma~4.1]{GangboKimPacini2011}, it follows that
\[
W_2\bigl((m_\pm)_\#\pi_j,(m_\pm)_\#\pi\bigr)
\le W_2(\pi_j,\pi)\longrightarrow0.
\]

\end{proof}

We are now ready to prove Theorem~\ref{thm:entropy-min-max-general} in the regular case.

\begin{lem}
\label{lem:entropy-min-max-regular}
Let $K\subset\mathbb R^d$ be a bounded, smooth, uniformly convex domain. Suppose that $\nu_1$ and $\nu_2$ have densities smooth on $\overline K$ that are bounded above and below by positive constants. Then
\[
H(\nu_-\mid\lambda)+H(\nu_+\mid\lambda)
\leq
H(\nu_1\mid\lambda)+H(\nu_2\mid\lambda).
\]
\end{lem}

\begin{proof}
We use Caffarelli's global regularity theory for the second boundary
value problem for the Monge--Ampère equation
\cite{Caf92,Caf96}.  The hypotheses of the result used here are that
the source and target supports are bounded uniformly convex domains
with $C^2$ boundaries and, for some $\alpha\in(0,1)$, that their
densities extend to functions in $C^{0,\alpha}$ on the closures of the
respective domains and satisfy two-sided bounds
\[
0<c\leq n_i\leq C<\infty,
\qquad i\in\{1,2\}.
\]
These hypotheses hold here, with both source and target domains equal
to $K$, because $K$ has smooth boundary and the densities are smooth on $\overline K$ and
bounded above and below by positive constants.  The cited regularity
theory therefore gives $\varphi\in C^{2,\alpha}(\overline K)$.  The
Monge--Amp\`ere equation and the two-sided density bounds imply that
$\det D^2\varphi>0$; since $\varphi$ is convex, $D^2\varphi$ is therefore
positive definite. In particular,
\[
T=\nabla\varphi\colon K\longrightarrow K
\]
is a diffeomorphism and
\[
A(x)\coloneqq DT(x)=D^2\varphi(x)
\]
is symmetric and positive definite. Since $K$ is convex, integration
of $D^2\varphi$ along line segments shows that
\[
\langle T(x)-T(y),x-y\rangle>0
\qquad (x\neq y),
\]
so Lemma~\ref{lem:min-max-injective} applies.

Consequently, the maps
\[
T_-(x)=x\wedge T(x),
\qquad
T_+(x)=x\vee T(x)
\]
are injective.

Fix a point $x$ at which $T_-$ and $T_+$ are differentiable, and set
\[
I=I(x)=\{i:x_i\leq T_i(x)\},
\qquad
J=I^c.
\]
At almost every such point, the Jacobian of $T_-$ is obtained by taking
the $i$-th row from the identity matrix for $i\in I$ and the $i$-th
row from $A(x)$ for $i\in J$. On the set where $x_i=T_i(x)$, the two
rows agree almost everywhere, since the gradient of the $C^1$ function
$x_i-T_i(x)$ vanishes almost everywhere on its zero set. Thus the
choice made in the definition of $I$ causes no ambiguity. After
reordering the coordinates so that
the indices in $I$ come first,
\[
DT_-(x)
=
\begin{pmatrix}
\operatorname{Id}_I&0\\
A_{JI}(x)&A_{JJ}(x)
\end{pmatrix}.
\]
Similarly,
\[
DT_+(x)
=
\begin{pmatrix}
A_{II}(x)&A_{IJ}(x)\\
0&\operatorname{Id}_J
\end{pmatrix}.
\]
Consequently,
\begin{equation}\label{eq:principal-minors}
\det DT_-(x)=\det A_{JJ}(x),
\qquad
\det DT_+(x)=\det A_{II}(x),
\end{equation}
where the determinant of an empty principal minor is understood to be
equal to $1$.

Since $A(x)$ is positive definite, Fischer's determinant inequality
gives
\[
\det A(x)
\leq
\det A_{II}(x)\det A_{JJ}(x).
\]
Together with \eqref{eq:principal-minors}, this yields
\begin{equation}\label{eq:determinant-min-max}
\det DT(x)
\leq
\det DT_-(x)\det DT_+(x)
\end{equation}
for almost every $x$.

Let $n_1,n_2,n_-,n_+$ denote the densities of
$\nu_1,\nu_2,\nu_-,\nu_+$, respectively. Since $T_-$ and $T_+$ are
injective, the area formula gives
\[
n_-(T_-(x))\det DT_-(x)=n_1(x)
\]
and
\[
n_+(T_+(x))\det DT_+(x)=n_1(x)
\]
for almost every $x$. Therefore,
\begin{align}
H(\nu_-\mid\lambda)
&=
H(\nu_1\mid\lambda)
-
\int_K\log\det DT_-(x)\,d\nu_1(x),
\label{eq:entropy-minus-regular}\\
H(\nu_+\mid\lambda)
&=
H(\nu_1\mid\lambda)
-
\int_K\log\det DT_+(x)\,d\nu_1(x).
\label{eq:entropy-plus-regular}
\end{align}
Adding these identities gives
\begin{align}
H(\nu_-\mid\lambda)+H(\nu_+\mid\lambda)
&=
2H(\nu_1\mid\lambda)\nonumber\\
&\quad-
\int_K
\log\left(
\det DT_-(x)\det DT_+(x)
\right)\,d\nu_1(x).
\label{eq:entropy-min-max-regular-sum}
\end{align}

The Monge--Ampère change-of-variables identity
\[
n_2(T(x))\det DT(x)=n_1(x)
\]
similarly gives
\begin{equation}\label{eq:entropy-target-regular}
H(\nu_2\mid\lambda)
=
H(\nu_1\mid\lambda)
-
\int_K\log\det DT(x)\,d\nu_1(x).
\end{equation}
Hence
\begin{align}
H(\nu_1\mid\lambda)+H(\nu_2\mid\lambda)
&=
2H(\nu_1\mid\lambda)
-
\int_K\log\det DT(x)\,d\nu_1(x).
\label{eq:entropy-input-regular-sum}
\end{align}

By \eqref{eq:determinant-min-max},
\[
\log\left(
\det DT_-(x)\det DT_+(x)
\right)
\geq
\log\det DT(x).
\]
Comparing \eqref{eq:entropy-min-max-regular-sum} with
\eqref{eq:entropy-input-regular-sum} proves the result.
\end{proof}

We can now remove the regularity assumptions to prove the theorem in full generality.

\begin{proof}[Proof of Theorem~\ref{thm:entropy-min-max-general}]
For an absolutely continuous measure in $\mathcal P_2(\mathbb R^d)$,
the entropy relative to Lebesgue measure is bounded below by a finite
quantity. Thus
$H(\nu\mid\lambda)\in(-\infty,+\infty]$.
If
$H(\nu_1\mid\lambda)+H(\nu_2\mid\lambda)=+\infty$,
the desired inequality is automatic. We may therefore assume that
both terms on the right-hand side are finite.

For every absolutely continuous $\nu\in\mathcal P_2(\mathbb R^d)$,
\begin{equation}\label{eq:gaussian-lebesgue-entropy}
H(\nu\mid\gamma)
=
H(\nu\mid\lambda)
+
\frac12\int_{\mathbb R^d}|x|^2\,d\nu(x)
+
\frac d2\log(2\pi).
\end{equation}
It follows that
$H(\nu_1\mid\gamma)+H(\nu_2\mid\gamma)<\infty$.

Choose the approximating measures $\nu_1^k,\nu_2^k$ supplied by
Lemma~\ref{lem:regular-entropy-approximation}, and let $\pi_k$ be their
quadratic optimal couplings. Define
$\nu_-^k=(m_-)_\#\pi_k$ and $\nu_+^k=(m_+)_\#\pi_k$.
By Lemma~\ref{lem:entropy-min-max-regular},
\begin{equation}\label{eq:regular-approximate-inequality}
H(\nu_-^k\mid\lambda)+H(\nu_+^k\mid\lambda)
\leq
H(\nu_1^k\mid\lambda)+H(\nu_2^k\mid\lambda).
\end{equation}

Since $(x_i\wedge y_i)^2+(x_i\vee y_i)^2=x_i^2+y_i^2$ for each coordinate $i$,
we have
$|x\wedge y|^2+|x\vee y|^2=|x|^2+|y|^2$ for each $x,y\in\mathbb R^d$.
Integrating against $\pi_k$ gives
\begin{align}
\int_{\mathbb R^d}|z|^2\,d\nu_-^k(z)
+
\int_{\mathbb R^d}|z|^2\,d\nu_+^k(z)
=
\int_{\mathbb R^d}|x|^2\,d\nu_1^k(x)
+
\int_{\mathbb R^d}|y|^2\,d\nu_2^k(y).
\label{eq:approximate-moment-identity}
\end{align}
Combining \eqref{eq:gaussian-lebesgue-entropy} and \eqref{eq:regular-approximate-inequality} (and observing the cancellation of the second-moment terms due to \eqref{eq:approximate-moment-identity}), we obtain
\begin{equation}\label{eq:regular-gaussian-inequality}
H(\nu_-^k\mid\gamma)+H(\nu_+^k\mid\gamma)
\leq
H(\nu_1^k\mid\gamma)+H(\nu_2^k\mid\gamma).
\end{equation}

Since $\nu_1$ is absolutely continuous, the quadratic optimal coupling
between $\nu_1$ and $\nu_2$ is unique. Lemma
\ref{lem:min-max-transport-stability} gives
\[
\nu_-^k\longrightarrow\nu_-,
\qquad
\nu_+^k\longrightarrow\nu_+
\]
in $W_2$.

Relative entropy with respect to the probability measure $\gamma$ is
lower semicontinuous under weak convergence. Hence
\begin{align*}
H(\nu_-\mid\gamma)+H(\nu_+\mid\gamma)
&\leq
\liminf_{k\to\infty}
\left[
H(\nu_-^k\mid\gamma)+H(\nu_+^k\mid\gamma)
\right]\\
&\leq
\lim_{k\to\infty}
\left[
H(\nu_1^k\mid\gamma)+H(\nu_2^k\mid\gamma)
\right]\\
&=
H(\nu_1\mid\gamma)+H(\nu_2\mid\gamma).
\end{align*}

Integrating the pointwise identity
$|x\wedge y|^2+|x\vee y|^2=|x|^2+|y|^2$ against $\pi$ gives
\[
\int |z|^2\,d\nu_-(z)+\int |z|^2\,d\nu_+(z)
=
\int |x|^2\,d\nu_1(x)+\int |y|^2\,d\nu_2(y).
\]
Substituting \eqref{eq:gaussian-lebesgue-entropy} for each of the four
Gaussian relative entropies therefore cancels the two quadratic-moment
terms and the two constants on each side, leaving
\[
H(\nu_-\mid\lambda)+H(\nu_+\mid\lambda)
\leq
H(\nu_1\mid\lambda)+H(\nu_2\mid\lambda).
\]
\end{proof}

\section{Connections with Pr\'ekopa--Leindler in the continuous setting} \label{sec: ADI PLI}

\subsection{The four-function inequalities of Ahlswede--Daykin and Cordero-Erausquin--Maurey}

First we present a result of Cordero-Erausquin and Maurey \cite{CM17:1}, which may be thought of as a four-function variant
of the Pr\'ekopa--Leindler inequality.

\begin{thm}[Cordero-Erausquin \& Maurey] \label{thm:cordero-maurey}
    Let $f_i: \mathbb{R}^d \to [0,\infty)$ be Borel measurable and integrable, and let $\lambda \in (0,1)$. If
    \[
        f_1(x) f_2(y) \leq f_3(\lambda x + (1-\lambda) y) f_4((1-\lambda) x + \lambda y),
    \]
    then
    \[
        \int f_1 \int f_2  \leq \int f_3 \int f_4.
    \]
\end{thm}

The preceding theorem is the specialization $r=1/2$ and $t=1-s$ of \cite[Proposition 4.6]{CM17:1}, in the notation of that proposition;
also \cite[Proposition 1.2]{CM17:1} is the particular case with coefficients $2/3$ and $1/3$. These results follow from the more general 
linear-operator framework of \cite[Theorem 1.6]{CM17:1}. For example, the linear map in Theorem \ref{thm:cordero-maurey} 
has block-matrix representation\footnote{Each scalar entry displayed in the matrix represents that scalar times the identity on $\mathbb{R}^d$.} 
$A = \begin{pmatrix}
\lambda & 1-\lambda \\
1-\lambda & \lambda
\end{pmatrix}$, as $A \begin{pmatrix} x \\ y \end{pmatrix} = \begin{pmatrix}
    \lambda x + (1-\lambda) y \\
    (1-\lambda) x + \lambda y 
\end{pmatrix}$, with outputs corresponding to the inputs of $f_3$ and $f_4$.

In their paper \cite{CM17:1}, Cordero-Erausquin and Maurey ask whether there is a transport proof of the preceding theorem.
We give below a positive answer to this question, as part of a more general framework involving 
entropy inequalities for coordinatewise power means under the quadratic 
optimal transport coupling in $\mathbb R^d$.

The generalization that we pursue here concerns nonlinear means of $x$ and $y$, with the aim of connecting the Ahlswede--Daykin four-function theorem with that of Cordero-Erausquin and Maurey.
To this end recall the definition of the generalized mean from the introduction. For $x,y > 0$, $\lambda \in (0,1)$ and $\alpha \in \mathbb{R}\cup\{\pm \infty\}$, define
$$
M_\alpha^\lambda(x,y) := \left( \lambda x^\alpha + (1-\lambda)y^\alpha \right)^\frac{1}{\alpha} 
$$
where the special cases $\alpha= \pm\infty, 0$ are to be understood in the limit, namely
$M_{-\infty}^\lambda(x,y)=\lim_{\alpha \to - \infty}M_\alpha^\lambda(x,y) = \min(x,y)$, similarly 
$M_{\infty}^\lambda(x,y)=\lim_{\alpha \to \infty}M_\alpha^\lambda(x,y) = \max(x,y)$ and 
$M_0^\lambda(x,y)=x^\lambda y^{1-\lambda}$ is the geometric mean. The case $\alpha=-1$ is the harmonic mean, while $\alpha=1$ corresponds to the arithmetic mean. 
    For vectors $x,y \in (0,\infty)^d$ we define a vector $M_\alpha^{\lambda}(x,y)$ by
    \[
        {M}_\alpha^{\lambda}(x,y)_i = {M}_\alpha^{\lambda}(x_i, y_i)
    \]
With these definitions in place we can state the two theorems together.

\begin{thm}[Ahlswede \& Daykin, Cordero-Erausquin \& Maurey] \label{thm: ADCEM}
    Let  $t \in (0,1)$, and $f_i: (0,\infty)^d \to [0,\infty)$ be integrable functions such that
    \[
        f_1(x)f_2(y) \leq f_3(M^t_\alpha(x,y)) f_4(M_\beta^{1-t}(x,y))
    \]
for        $(\alpha,\beta)=(1,1)$ or $(\alpha,\beta)=(-\infty,\infty)$.
Then
    \[
        \int f_1 \int f_2 \leq \int f_3 \int f_4.
    \]
\end{thm}
\begin{rem}
For either of the two parameter pairs in the theorem, restricting the
domain to $(0,\infty)^d$ entails no loss relative to the corresponding
statement on $\mathbb R^d$. To see this, suppose that
$f_i\colon\mathbb R^d\to[0,\infty)$ satisfy the pointwise hypothesis.
Let
\[
Q_R=[-R,R]^d,
\qquad a_R=(R+1)(1,\ldots,1),
\]
and, for $z\in(0,\infty)^d$, define
\[
g_{i,R}(z)=f_i(z-a_R)\mathbbm{1}_{Q_R}(z-a_R),\quad i=1,2,
\qquad
g_{i,R}(z)=f_i(z-a_R),\quad i=3,4.
\]
The arithmetic-mean and min--max maps commute with common
translations. Hence the $g_{i,R}$ satisfy the positive-orthant
hypothesis: if either translated input lies outside $Q_R$, its
left-hand side is zero, while if both lie in $Q_R$, both outputs remain
in $a_R+Q_R\subset(0,\infty)^d$ and the original hypothesis applies.
The positive-orthant conclusion therefore gives
\[
\left(\int_{Q_R}f_1\right)\left(\int_{Q_R}f_2\right)
\leq
\left(\int_{(-a_R,\infty)^d}f_3\right)
\left(\int_{(-a_R,\infty)^d}f_4\right)
\leq
\left(\int_{\mathbb R^d}f_3\right)
\left(\int_{\mathbb R^d}f_4\right).
\]
Letting $R\to\infty$ and using monotone convergence proves the
full-space statement.
\end{rem}

In the sequel we consider the following question.  
\par\vspace{.1in}
\noindent{\it Question}: Are there other pairs of $(\alpha, \beta)$ such that Theorem \ref{thm: ADCEM} continues to hold as stated?
\par\vspace{.1in}

To address this question, we first give a scalar criterion and its transport proof. We then obtain a weighted extension for orders in $[0,1]$ and the sharp range $0\leq\alpha+\beta\leq2$ when both means have equal weights.

\subsection{A scalar criterion and its transport proof}

For $q\in\mathbb R$ and $\theta\in(0,1)$, write
\begin{equation}\label{eq:scalar-power-jacobian}
D_{q,\theta}(x,y;z)
=
M_q^\theta(x,y)^{1-q}
\left[\theta x^{q-1}+(1-\theta)y^{q-1}z\right],
\qquad x,y,z>0.
\end{equation}
Thus $D_{q,\theta}=\partial_1M_q^\theta+z\partial_2M_q^\theta$,
also when $q=0$. The following determinant bound allows us to pass
from a scalar differential inequality to the transport Jacobians.

\begin{lem}\label{lem:diagonal-jacobian-bound}
Let $A$ be a positive definite symmetric $d\times d$ matrix, and write
\[
A=L\operatorname{diag}(\delta_1,\ldots,\delta_d)L^{\mathsf T},
\]
where $L$ is unit lower triangular and $\delta_i>0$. For positive
diagonal matrices $U=\operatorname{diag}(u_i)$ and
$V=\operatorname{diag}(v_i)$,
\[
\det(U+VA)\geq\prod_{i=1}^d(u_i+v_i\delta_i).
\]
\end{lem}

\begin{proof}
For $I\subseteq\{1,\ldots,d\}$, let $A_I$ be the corresponding
principal submatrix. Cauchy--Binet gives
\[
\det A_I
=
\sum_{|J|=|I|}(\det L_{I,J})^2\prod_{j\in J}\delta_j
\geq\prod_{i\in I}\delta_i,
\]
because the summand with $J=I$ has $\det L_{I,I}=1$.
Expanding the determinant by multilinearity now gives
\[
\det(U+VA)
=
\sum_I\det A_I\prod_{i\in I}v_i\prod_{i\notin I}u_i
\geq
\prod_{i=1}^d(u_i+v_i\delta_i).
\]
Here the empty principal minor and empty products are understood to
equal one.
\end{proof}

We will also use the following compact-support form of the Gibbs
variational formula.

\begin{lem}[Restricted Gibbs variational formula]
\label{lem:restricted-gibbs}
Let $f\colon(0,\infty)^d\to[0,\infty)$ be measurable with
$0<\int_{(0,\infty)^d} f\,d\lambda<\infty$. Then
\begin{equation}\label{eq:restricted-gibbs}
\log\int_{(0,\infty)^d} f\,d\lambda
=
\sup_\nu\left\{\int\log f\,d\nu-H(\nu\mid\lambda)\right\},
\end{equation}
where the supremum may be restricted to absolutely continuous
probability measures with compact support and finite entropy that are
concentrated on $\{f>0\}$ and satisfy $\log f\in L^1(\nu)$.
\end{lem}

\begin{proof}
The Gibbs variational formula gives the upper bound. For the reverse
bound, set
\[
E_k=\{x\in[k^{-1},k]^d:k^{-1}\leq f(x)\leq k\},
\qquad Z_k=\int_{E_k}f\,d\lambda.
\]
The sets $E_k$ increase to $\{0<f<\infty\}$, up to a null set, so
$Z_k\uparrow\int f\,d\lambda$. Whenever $Z_k>0$, define
\[
d\nu_k=Z_k^{-1}f\mathbbm{1}_{E_k}\,d\lambda.
\]
Then $\nu_k$ is admissible and a direct calculation gives
\[
\int\log f\,d\nu_k-H(\nu_k\mid\lambda)=\log Z_k.
\]
Letting $k\to\infty$ proves the claim.
\end{proof}

\begin{thm}[A scalar criterion for power-mean inequalities]
\label{thm:scalar-jacobian-entropy}
Let $\alpha,\beta\in\mathbb R$, let $r,s,t\in(0,1)$, and set
$m=rs+(1-r)t$. Suppose that
\begin{equation}\label{eq:scalar-jacobian-condition}
D_{\alpha,s}(x,y;z)^rD_{\beta,t}(x,y;z)^{1-r}
\geq z^{1-m}
\qquad\text{for all }x,y,z>0.
\end{equation}
Let $\nu_1,\nu_2\in\mathcal P_2((0,\infty)^d)$ be absolutely
continuous probability measures with finite entropy relative to
Lebesgue measure $\lambda$, and let
$\pi=(\operatorname{Id},T)_\#\nu_1$ be their Brenier coupling. Define
\[
\nu_3=(M_\alpha^s)_\#\pi,\qquad
\nu_4=(M_\beta^t)_\#\pi.
\]
Then
\begin{equation}\label{eq:scalar-jacobian-entropy}
rH(\nu_3\mid\lambda)+(1-r)H(\nu_4\mid\lambda)
\leq mH(\nu_1\mid\lambda)+(1-m)H(\nu_2\mid\lambda).
\end{equation}
Moreover, if nonnegative integrable functions $f_1,f_2,f_3,f_4$
on $(0,\infty)^d$ satisfy
\begin{equation}\label{eq:four-function-power-mean-assumption}
f_1(x)^m f_2(y)^{1-m}
\leq
f_3\bigl(M_\alpha^s(x,y)\bigr)^r
f_4\bigl(M_\beta^t(x,y)\bigr)^{1-r}
\end{equation}
for all $x,y\in(0,\infty)^d$, then
\begin{equation}\label{eq:four-function-power-mean-conclusion}
\left(\int_{(0,\infty)^d}f_1\right)^m
\left(\int_{(0,\infty)^d}f_2\right)^{1-m}
\leq
\left(\int_{(0,\infty)^d}f_3\right)^r
\left(\int_{(0,\infty)^d}f_4\right)^{1-r}.
\end{equation}
The scalar condition \eqref{eq:scalar-jacobian-condition} is necessary
for either assertion to hold for all admissible inputs, already in
dimension one.
\end{thm}

\begin{proof}
\emph{Entropy inequality for regular densities.}
We first assume that $\nu_1,\nu_2$ have densities smooth on the
closure and bounded above and below by positive constants on a
common bounded, smooth, uniformly convex domain
$\Omega\Subset(0,\infty)^d$. Under these assumptions the global
regularity theory for the second boundary value problem for the
Monge--Amp\`ere equation \cite{Caf92,Caf96}, used also in
Lemma~\ref{lem:entropy-min-max-regular}, gives a $C^1$
diffeomorphism $T\colon\Omega\to\Omega$ for which
$A(x):=DT(x)=D^2\varphi(x)$ is symmetric and positive definite.

Put $(q_3,\theta_3)=(\alpha,s)$ and
$(q_4,\theta_4)=(\beta,t)$, and set
\[
G_j(x)=M_{q_j}^{\theta_j}(x,T(x)),\qquad j=3,4.
\]
Differentiating coordinatewise gives $DG_j=U_j+V_jA$, where
\[
(U_j)_{ii}=\partial_1M_{q_j}^{\theta_j}(x_i,T_i(x)),
\qquad
(V_j)_{ii}=\partial_2M_{q_j}^{\theta_j}(x_i,T_i(x)).
\]
Both diagonal matrices are positive. Apply
Lemma~\ref{lem:diagonal-jacobian-bound} to both Jacobians, using the
same pivots $\delta_i$ of $A$. By
\eqref{eq:scalar-jacobian-condition},
\begin{align}
(\det DG_3)^r(\det DG_4)^{1-r}
&\geq
\prod_{i=1}^d
D_{\alpha,s}(x_i,T_i(x);\delta_i)^r
D_{\beta,t}(x_i,T_i(x);\delta_i)^{1-r}\nonumber\\
&\geq
\prod_{i=1}^d\delta_i^{1-m}
=(\det DT)^{1-m}.
\label{eq:scalar-jacobian-determinant}
\end{align}

The maps $G_3,G_4$ are injective. Indeed, positive definiteness of
$DT$ and convexity of $\Omega$ imply strict monotonicity of $T$.
If $G_j(x)=G_j(y)$, strict increase of each scalar mean in both
arguments gives
\[
(x_i-y_i)(T_i(x)-T_i(y))\leq0
\]
in every coordinate. Summing contradicts strict monotonicity unless
$x=y$. Thus the change-of-variables formulas give
\[
H(\nu_j\mid\lambda)
=
H(\nu_1\mid\lambda)-\int\log\det DG_j\,d\nu_1,
\qquad j=3,4,
\]
and
\[
H(\nu_2\mid\lambda)
=
H(\nu_1\mid\lambda)-\int\log\det DT\,d\nu_1.
\]
Taking logarithms in \eqref{eq:scalar-jacobian-determinant},
integrating, and using these identities proves
\eqref{eq:scalar-jacobian-entropy} in the regular case.

\emph{Approximation.} We now remove the regularity assumptions. By
\eqref{eq:gaussian-lebesgue-entropy}, the input measures have finite
entropy relative to the standard Gaussian. The construction in
Lemma~\ref{lem:regular-entropy-approximation} can be carried out
inside the positive orthant, using the common exhaustion
\[
\Omega_R=
\left\{x\in(0,\infty)^d:
\sum_{i=1}^d(x_i+x_i^{-1})<R\right\},\qquad R>2d.
\]
These domains are bounded, smooth and uniformly convex: the defining
function has positive definite Hessian, and its level sets are compact
and regular. First restrict and normalize on a smaller member of the
exhaustion. Truncate the resulting densities from above, mollify
within the larger domain, and mix with a vanishing uniform density
there. Choosing each error successively as in the cited lemma gives
densities smooth on the closures and bounded above and below by
positive constants on common domains, such that
\begin{equation}\label{eq:coordinatewise-power-approximation}
\nu_i^k\longrightarrow\nu_i\quad\text{in }W_2,\qquad
H(\nu_i^k\mid\lambda)\longrightarrow H(\nu_i\mid\lambda),
\quad i=1,2.
\end{equation}
The entropy convergence follows from the Gaussian entropy convergence
in that lemma and the second-moment identity
\eqref{eq:gaussian-lebesgue-entropy}.

Let $\pi_k$ be the Brenier coupling between $\nu_1^k$ and $\nu_2^k$.
Lemma~\ref{lem:min-max-transport-stability} gives
$\pi_k\to\pi$ in $W_2$. Every power mean of finite real order with
weights in $(0,1)$ extends continuously to the closed positive
orthant and lies between the coordinatewise minimum and maximum.
Consequently its growth is at most linear, and
\[
(M_\alpha^s)_\#\pi_k\longrightarrow\nu_3,\qquad
(M_\beta^t)_\#\pi_k\longrightarrow\nu_4
\quad\text{in }W_2.
\]
Entropy relative to Lebesgue measure is lower semicontinuous under
$W_2$ convergence, by \eqref{eq:gaussian-lebesgue-entropy} and
convergence of second moments. The same identity bounds the output
entropies below uniformly. We may therefore pass to the limit in the
weighted entropy inequality, proving
\eqref{eq:scalar-jacobian-entropy}.

\emph{The functional inequality.}
Write $Q=(0,\infty)^d$. It is enough to consider the case when
$0<\int_Q f_i\,d\lambda<\infty$ for $i\in\{1,2,3,4\}$. If either
input integral vanishes the conclusion is immediate. If both input
integrals are positive, fix $y\in\{f_2>0\}$ and a positive-measure
subset $E$ of $\{f_1>0\}$ contained in a compact box $K\subset Q$.
For this fixed $y$, each map $x\mapsto M_q^\theta(x,y)$ is
bi-Lipschitz on $K$ when $q\in\mathbb R$ and $\theta\in(0,1)$: its
coordinate derivatives are continuous and strictly positive there.
The pointwise hypothesis forces both $f_3$ and $f_4$ to be positive
on the respective images of $E$, which have positive measure.
Thus both output integrals are positive as well.

Let $\nu_1,\nu_2$ be absolutely continuous probability measures with
compact support and finite entropy, chosen so that $\nu_i$ is
concentrated on $\{f_i>0\}$ and $\log f_i\in L^1(\nu_i)$, for
$i\in\{1,2\}$. Let $\pi$ be their Brenier coupling, and define
\[
\nu_3=(m_3)_\#\pi,
\qquad
\nu_4=(m_4)_\#\pi,
\]
where
\[
m_3(x,y)=M_\alpha^s(x,y),
\qquad
m_4(x,y)=M_\beta^t(x,y).
\]

By the entropy inequality just proved,
\[
rH(\nu_3\mid\lambda)
+
(1-r)H(\nu_4\mid\lambda)
\leq
mH(\nu_1\mid\lambda)
+
(1-m)H(\nu_2\mid\lambda).
\]
The logarithmic integrals below are well defined. Indeed, the output
measures have compact support and finite entropy, so the Gibbs bound
on a compact set, applied to $\min\{\log(1+f_j),k\}$ and then letting $k\to\infty$, makes their positive logarithmic parts integrable. The pointwise hypothesis and the
integrability of $\log f_1$ and $\log f_2$ then give integrability of
the negative parts.

Taking logarithms in
\eqref{eq:four-function-power-mean-assumption} and integrating against
$\pi$ gives
\begin{align}
m\int\log f_1\,d\nu_1
+
(1-m)\int\log f_2\,d\nu_2
&\leq
r\int\log f_3\,d\nu_3
+
(1-r)\int\log f_4\,d\nu_4.
\label{eq:integrated-four-function-assumption}
\end{align}
Combining \eqref{eq:scalar-jacobian-entropy} with
\eqref{eq:integrated-four-function-assumption}, we obtain
\begin{align}
&m\left[
\int\log f_1\,d\nu_1-H(\nu_1\mid\lambda)
\right]
+
(1-m)\left[
\int\log f_2\,d\nu_2-H(\nu_2\mid\lambda)
\right]\nonumber\\
&\leq
r\left[
\int\log f_3\,d\nu_3-H(\nu_3\mid\lambda)
\right]
+
(1-r)\left[
\int\log f_4\,d\nu_4-H(\nu_4\mid\lambda)
\right].
\label{eq:four-function-dual-entropy}
\end{align}
By the Gibbs variational formula, the right-hand side is bounded above by
\[
r\log\int_Q f_3\,d\lambda
+
(1-r)\log\int_Q f_4\,d\lambda.
\]
Since $\nu_1$ and $\nu_2$ are otherwise arbitrary, taking their suprema independently and using Lemma~\ref{lem:restricted-gibbs}
gives
\begin{align*}
m\log\int_Q f_1\,d\lambda
+
(1-m)\log\int_Q f_2\,d\lambda
&\leq
r\log\int_Q f_3\,d\lambda
+
(1-r)\log\int_Q f_4\,d\lambda.
\end{align*}
Exponentiating proves
\[
\left(\int_Q f_1\,d\lambda\right)^m
\left(\int_Q f_2\,d\lambda\right)^{1-m}
\leq
\left(\int_Q f_3\,d\lambda\right)^r
\left(\int_Q f_4\,d\lambda\right)^{1-r}.
\]

\emph{Necessity of the scalar condition.} Fix $x,y,z>0$ and take the input measures to be uniform
on
\[
I_\varepsilon=[x-\varepsilon,x+\varepsilon],\qquad
J_\varepsilon=[y-z\varepsilon,y+z\varepsilon],
\]
where $\varepsilon>0$ is small enough that both intervals are
positive. Their Brenier map is $T(u)=y+z(u-x)$. The
change-of-variables formulas reduce the entropy assertion to
\[
\frac{1}{2\varepsilon}\int_{I_\varepsilon}
\left[r\log D_{\alpha,s}(u,T(u);z)
+(1-r)\log D_{\beta,t}(u,T(u);z)\right]\,du
\geq(1-m)\log z.
\]
Letting $\varepsilon\downarrow0$ gives
\eqref{eq:scalar-jacobian-condition}.

For the functional assertion, take $f_1,f_2$ to be the indicators of
$I_\varepsilon,J_\varepsilon$, and $f_3,f_4$ the indicators of the
images of $I_\varepsilon\times J_\varepsilon$ under
$M_\alpha^s,M_\beta^t$, respectively. These images are
intervals, since the means are continuous and strictly increasing.
Their lengths are
\[
2\varepsilon D_{\alpha,s}(x,y;z)+o(\varepsilon),
\qquad
2\varepsilon D_{\beta,t}(x,y;z)+o(\varepsilon).
\]
The pointwise hypothesis holds by construction. Dividing the claimed
integral inequality by $2\varepsilon$ and letting
$\varepsilon\downarrow0$ again gives
\eqref{eq:scalar-jacobian-condition}. 
\end{proof}

\subsection{Power-mean consequences}

\begin{cor}[Entropy inequality for coordinatewise power means]
\label{thm:coordinatewise-power-entropy}
Let $\nu_1,\nu_2$ and $\pi$ be as in
Theorem~\ref{thm:scalar-jacobian-entropy}. For $q,\theta\in[0,1]$, put
$\nu_q^\theta=(M_q^\theta)_\#\pi$. Then
\begin{equation}\label{eq:coordinatewise-power-entropy}
H(\nu_q^\theta\mid\lambda)
\leq\theta H(\nu_1\mid\lambda)+(1-\theta)H(\nu_2\mid\lambda).
\end{equation}
Consequently, for $\alpha,\beta,r,s,t\in[0,1]$ and $m=rs+(1-r)t$,
\begin{equation}\label{eq:coordinatewise-power-entropy-combined}
rH((M_\alpha^s)_\#\pi\mid\lambda)
+(1-r)H((M_\beta^t)_\#\pi\mid\lambda)
\leq mH(\nu_1\mid\lambda)+(1-m)H(\nu_2\mid\lambda).
\end{equation}
\end{cor}

\begin{proof}
For $\theta\in(0,1)$ and $q\in[0,1]$, the weighted
arithmetic-geometric mean inequality gives
\[
D_{q,\theta}(x,y;z)
\geq
M_q^\theta(x,y)^{1-q}
(x^\theta y^{1-\theta})^{q-1}z^{1-\theta}
\geq z^{1-\theta},
\]
because $M_q^\theta(x,y)\geq x^\theta y^{1-\theta}$ and $1-q\geq0$.
Apply Theorem~\ref{thm:scalar-jacobian-entropy} with
$\alpha=\beta=q$, $s=t=\theta$, and $r=1/2$ to obtain
\eqref{eq:coordinatewise-power-entropy}. The cases $\theta=0,1$
are immediate. Applying this inequality to the two output means,
multiplying by $r$ and $1-r$, and adding proves
\eqref{eq:coordinatewise-power-entropy-combined}.
\end{proof}

\begin{cor}[Weighted four-function inequality]\label{thm:cordero-maurey-general}
Let $r,s,t\in(0,1)$, let $m=sr+(1-r)t$, and let $\alpha,\beta\in[0,1]$. If integrable functions
$f_i\colon(0,\infty)^d\to[0,\infty)$ satisfy
\eqref{eq:four-function-power-mean-assumption} for all $x,y\in(0,\infty)^d$, then
\eqref{eq:four-function-power-mean-conclusion} holds.
\end{cor}

\begin{proof}
The bound in the proof of Corollary~\ref{thm:coordinatewise-power-entropy} gives
\[
D_{\alpha,s}(x,y;z)^rD_{\beta,t}(x,y;z)^{1-r}
\geq z^{r(1-s)+(1-r)(1-t)}=z^{1-m}.
\]
The conclusion follows from Theorem~\ref{thm:scalar-jacobian-entropy}.
\end{proof}

\begin{rem}
Observe that taking $r = \frac 1 2$ and $s = 1-t$ forces $m = \frac 1 2$ and we see that Corollary~\ref{thm:cordero-maurey-general} extends Theorem \ref{thm: ADCEM} to include $(\alpha, \beta) \in [0,1] \times[0,1]$.  
Note that taking $r \to 0$, one recovers a generalization of the usual three-function Pr\'ekopa--Leindler inequality.
\end{rem}

\begin{rem}
Note also that, in Slomka \cite[Theorem 1.2]{slomka2024remark}, a similar result for $\alpha=\beta=1$ is obtained in the discrete setting, but with exponents
$\alpha_1, \alpha_2, \alpha_3, \alpha_4$ associated to $f_1, f_2, f_3, f_4$ satisfying $\max(\alpha_1,\alpha_2) \leq \min(\alpha_3,\alpha_4)$, which, for us, would force $m=r=1/2$. 
Owing to the continuous setting's simpler computations we attain a more general formulation here and an interpolation between the usual three-function and a four-function variant of Pr\'ekopa--Leindler.  
In the discrete setting three-function versions of Pr\'ekopa--Leindler inequalities seem to be more delicate, see \cite{marsiglietti2024geometric} for a partial result on $\mathbb{Z}$ 
as well as \cite{malliaris2025functional} which reduces the functional problem to a geometric question.
For a four-function version of the Pr\'ekopa--Leindler inequality
on $\mathbb{Z}$ we refer the reader to \cite{KL19,GRST21,HKS21,slomka2024remark}.
\end{rem}

\begin{cor}[Equal-weight power means]
\label{cor:power-mean-strip}
Let $\nu_1,\nu_2$ and $\pi$ be as in
Theorem~\ref{thm:scalar-jacobian-entropy}. If
$\alpha,\beta\in\mathbb R$ and $0\leq\alpha+\beta\leq2$, then
\begin{equation}\label{eq:power-mean-strip-entropy}
H((M_\alpha^{1/2})_\#\pi\mid\lambda)
+H((M_\beta^{1/2})_\#\pi\mid\lambda)
\leq H(\nu_1\mid\lambda)+H(\nu_2\mid\lambda).
\end{equation}
Moreover, if nonnegative integrable functions on $(0,\infty)^d$
satisfy
\begin{equation}\label{eq:power-mean-strip-assumption}
f_1(x)f_2(y)
\leq f_3(M_\alpha^{1/2}(x,y))f_4(M_\beta^{1/2}(x,y)),
\end{equation}
then
\begin{equation}\label{eq:power-mean-strip-functional}
\left(\int f_1\right)\left(\int f_2\right)
\leq\left(\int f_3\right)\left(\int f_4\right).
\end{equation}
For finite orders, the range $0\leq\alpha+\beta\leq2$ is necessary
for either assertion to hold for all admissible inputs, already in
dimension one. Both assertions also hold for the limiting pairs
$(\alpha,\beta)=(-\infty,\infty)$ and $(\infty,-\infty)$.
\end{cor}

\begin{proof}
Take $r=s=t=1/2$ in Theorem~\ref{thm:scalar-jacobian-entropy}.
For fixed $x,y>0$, put
\[
a_q=\partial_1M_q^{1/2}(x,y),\qquad
b_q=\partial_2M_q^{1/2}(x,y).
\]
Since these derivatives are positive,
\begin{align*}
\inf_{z>0}\frac{D_{\alpha,1/2}(x,y;z)D_{\beta,1/2}(x,y;z)}{z}
&=\inf_{z>0}\left[
\frac{a_\alpha a_\beta}{z}
+a_\alpha b_\beta+b_\alpha a_\beta+b_\alpha b_\beta z
\right]\\
&=\left(\sqrt{a_\alpha b_\beta}+\sqrt{b_\alpha a_\beta}\right)^2.
\end{align*}
The derivatives are homogeneous of degree zero, so normalize $xy=1$
and write $x=e^{-u}$, $y=e^u$. Substitution shows that the logarithm
of the last expression is
\begin{equation}\label{eq:power-mean-strip-scalar}
K_{\alpha,\beta}(u)
=
\frac{1-\alpha}{\alpha}\log\cosh(\alpha u)
+\frac{1-\beta}{\beta}\log\cosh(\beta u)
+2\log\cosh\left(\frac{\alpha-\beta}{2}u\right).
\end{equation}
The terms at zero order are interpreted by continuity. Thus the
scalar condition is equivalent to $K_{\alpha,\beta}(u)\geq0$
for all $u\in\mathbb R$.

Fix $u\geq0$ and set
\[
g(v)=\frac{\log\cosh(vu)}{v},\qquad g(0)=0.
\]
This function is odd, and is increasing and concave on $[0,\infty)$.
Indeed, writing $w=uv$, one has
\[
v^2g'(v)=w\tanh w-\log\cosh w\geq0
\]
and
\[
v^3g''(v)=w^2\operatorname{sech}^2w-2w\tanh w
+2\log\cosh w\leq0.
\]
Both expressions on the right vanish at $w=0$, and their derivatives
are respectively $w\operatorname{sech}^2w\geq0$ and
$-2w^2\operatorname{sech}^2w\tanh w\leq0$.
Concavity and $g(0)=0$ also imply subadditivity on $[0,\infty)$.

By symmetry in the orders, write
$\alpha=c+h$, $\beta=c-h$, with $0\leq c\leq1$ and $h\geq0$.
Then \eqref{eq:power-mean-strip-scalar} becomes
\begin{align*}
K_{\alpha,\beta}(u)
&=(1-c)\bigl[g(c+h)+g(c-h)\bigr]\\
&\quad+h\bigl[2g(h)+g(c-h)-g(c+h)\bigr].
\end{align*}
The first bracket is nonnegative by monotonicity and oddness.
If $h\geq c$, concavity gives
$2g(h)\geq g(h-c)+g(h+c)$, so the second bracket is also
nonnegative. If $h\leq c$, subadditivity and concavity give
\[
g(c+h)\leq g(c-h)+g(2h)\leq g(c-h)+2g(h),
\]
with the same conclusion. Since $K_{\alpha,\beta}$ is even in $u$,
this proves the scalar condition for all $u$ and hence both
assertions for finite orders.

For necessity, writing $S=\alpha+\beta$, expansion at the origin gives
\[
K_{\alpha,\beta}(u)=\frac{S(2-S)}4u^2+O(u^4).
\]
If $S\notin[0,2]$, the scalar condition fails for small nonzero $u$.
The necessity assertion in
Theorem~\ref{thm:scalar-jacobian-entropy} proves sharpness.

Finally, let $\alpha\to-\infty$ with $\beta=1-\alpha$. The two
means converge pointwise to the coordinatewise minimum and maximum.
They are bounded in norm by $|x|+|y|$, so dominated convergence
under the fixed coupling $\pi$ gives $W_2$ convergence of their
marginals. Lower semicontinuity, with the same Gaussian entropy
identity used above, gives the limiting entropy inequality.
The Gibbs argument in the proof of
Theorem~\ref{thm: four function continous} applies to this inequality
and gives the functional endpoint; no continuity of the $f_i$ is
needed. The other limiting pair follows by exchanging the outputs.
\end{proof}

\begin{rem}\label{rem:power-mean-random-vectors}
Let $X,Y$ be random vectors taking values in $(0,\infty)^d$, with
absolutely continuous laws, finite second moments, and finite Shannon
entropies, coupled by quadratic optimal transport. The means below act
coordinatewise. In terms of Shannon entropy,
Corollary~\ref{cor:power-mean-strip} gives
\[
H(M_\alpha^{1/2}(X,Y))+H(M_\beta^{1/2}(X,Y))
\geq H(X)+H(Y)
\]
whenever $\alpha,\beta\in\mathbb R$ and $0\leq\alpha+\beta\leq2$.
Equivalently,
\[
N(M_\alpha^{1/2}(X,Y))N(M_\beta^{1/2}(X,Y))
\geq N(X)N(Y).
\]
More generally, under the scalar condition
\eqref{eq:scalar-jacobian-condition},
Theorem~\ref{thm:scalar-jacobian-entropy} gives the weighted inequality
\[
N(M_\alpha^s(X,Y))^r N(M_\beta^t(X,Y))^{1-r}
\geq N(X)^m N(Y)^{1-m},
\qquad m=rs+(1-r)t,
\]
for $r,s,t\in(0,1)$. For a single mean,
Corollary~\ref{thm:coordinatewise-power-entropy} gives
\[
N(M_q^\theta(X,Y))\geq N(X)^\theta N(Y)^{1-\theta},
\qquad q,\theta\in[0,1].
\]
As in Remark~\ref{rem:min-max-entropy-power}, these are multiplicative
entropy-power inequalities under the optimal transport coupling,
rather than the independent coupling. The min--max limit of the
equal-weight inequality recovers the inequality in that remark for
random vectors in the positive orthant.
\end{rem}

\begin{rem}
The equal-weight result includes the complementary orders
$(q,1-q)$ for every $q\in\mathbb R$. The choice $(\alpha,\beta)=(1,1)$
gives the midpoint case of the Cordero-Erausquin--Maurey inequality,
while the min--max limit gives Ahlswede--Daykin. The full weighted
Cordero-Erausquin--Maurey result follows instead from
Corollary~\ref{thm:coordinatewise-power-entropy} with
$\alpha=\beta=1$.
\end{rem}

\begin{rem}
The orders may also vary with the coordinate. For a vector
$q=(q_1,\ldots,q_d)$, define
$M_q^\theta(x,y)_i=M_{q_i}^\theta(x_i,y_i)$.
The determinant proof applies whenever
\eqref{eq:scalar-jacobian-condition} holds in each coordinate, with
the same $r,s,t$. Thus the weighted result allows
$\alpha_i,\beta_i\in[0,1]$, and the equal-weight result allows
$0\leq\alpha_i+\beta_i\leq2$ in every coordinate.
\end{rem}

{\bf{Acknowledgements:}} 
We warmly thank Hiroshi Tsuji for pointing out reference \cite{karlin-rinott80} that helped to change and improve the results and the presentation of our paper. 
The second author's work was supported in part by SECIHTI grant CBF-2024-2024-3907. While an initial draft of the paper was written without AI assistance, a commercial AI tool was used to work out details of the regularization arguments in the transport proofs, to develop the sharp conditions and proofs in Section 7, as well as for copy editing. The authors have verified all arguments and text themselves, and take complete responsibility for everything in the paper.

\bibliographystyle{imsart-number}
\bibliography{LSM}

\end{document}